\documentclass[a4paper,12pt]{article}

\usepackage[left=2cm,right=2cm, top=2cm,bottom=3cm,bindingoffset=0cm]{geometry}

\usepackage{verbatim}
\usepackage{amsmath}
\usepackage{amsthm}
\usepackage{amssymb}
\usepackage{delarray}
\usepackage{cite}
\usepackage{hyperref}
\usepackage{mathrsfs}
\usepackage{tikz}
\usetikzlibrary{patterns}
\usepackage{caption}
\DeclareCaptionLabelSeparator{dot}{. }
\newcommand{\al}{\alpha}

\newcommand{\ga}{\gamma}
\newcommand{\de}{\delta}
\newcommand{\la}{\lambda}

\newcommand{\eps}{\varepsilon}

\newcommand{\iy}{\infty}

\theoremstyle{plain}

\numberwithin{equation}{section}

\newtheorem{thm}{Theorem}[section]
\newtheorem{lem}[thm]{Lemma}
\newtheorem{prop}[thm]{Proposition}
\newtheorem{cor}[thm]{Corollary}

\theoremstyle{definition}

\newtheorem{alg}[thm]{Algorithm}
\newtheorem{ip}[thm]{Inverse Problem}

\theoremstyle{remark}

\newtheorem{remark}[thm]{Remark}

\DeclareMathOperator*{\Res}{Res}
\DeclareMathOperator{\rank}{rank}
\DeclareMathOperator{\diag}{diag}

 \allowdisplaybreaks

\begin{document}

\begin{center}
{\large\bf Constructive solution of the inverse spectral problem for the matrix Sturm-Liouville operator}
\\[0.2cm]
{\bf Natalia P. Bondarenko} \\[0.2cm]
\end{center}

\vspace{0.5cm}

{\bf Abstract.} An inverse spectral problem is studied for the matrix Sturm-Liouville operator on a finite interval with the general self-adjoint boundary condition.
We obtain a constructive solution based on the method of spectral mappings for the considered inverse problem. 
The nonlinear inverse problem is reduced to a linear equation in a special Banach space of infinite matrix sequences.
In addition, we apply our results to the Sturm-Liouville operator on a star-shaped graph.
  
\medskip


\medskip


\vspace{1cm}

\section{Introduction}

The main aim of this paper is to provide a constructive solution of the inverse spectral problem for the matrix Sturm-Liouville operator with the general self-adjoint boundary condition. The operator under consideration corresponds to the following eigenvalue problem $L = L(Q(x), T, H)$:
\begin{gather} \label{eqv}
    -Y'' + Q(x) Y = \la Y, \quad x \in (0, \pi), \\ \label{bc}
    Y(0) = 0, \quad V(Y) := T (Y'(\pi) - H Y(\pi)) - T^{\perp} Y(\pi) = 0.
\end{gather}
Here $Y(x) = [y_j(x)]_{j = 1}^m$ is an $m$-element vector function; $Q(x) = [q_{jk}(x)]_{j, k = 1}^m$ is an $m \times m$ matrix function, called the {\it potential}, $q_{jk} \in L_2(0, \pi)$; $\la$ is the spectral parameter.
Denote by $\mathbb C^m$ and $\mathbb C^{m \times m}$ the spaces of complex-valued $m$-element column vectors and $m \times m$ matrices, respectively.
It is supposed that $T \in \mathbb C^{m \times m}$ is an orthogonal projector in $\mathbb C^m$, $T^{\perp} = I_m - T$, where $I_m \in \mathbb C^{m \times m}$ is the unit matrix, and $H = T H T$. The matrices $Q(x)$ and $H$ are assumed to be Hermitian, i.e. $Q(x) = Q^{\dagger}(x)$ a.e. on $(0, \pi)$ and $H = H^{\dagger}$, where the symbol $\dagger$ denotes the conjugate transpose. Under these restrictions, the boundary value problem $L$ is self-adjoint.

We have imposed the boundary conditions~\eqref{bc}, since the problem \eqref{eqv}-\eqref{bc} generalizes eigenvalue problems for Sturm-Liouville operators on a star-shaped graph (see, e.g., \cite{Piv07, Bond19}). Differential operators on geometrical graphs, also called quantum graphs, have applications in mechanics, organic chemistry, mesoscopic physics, nanotechnology, theory of waveguides and other branches of science and engineering (see \cite{Nic85, LLS94, Kuch02, BCFK06, PP04} and references therein). 


Inverse problems of spectral analysis consist in reconstruction of operators, by using their spectral information. The most complete results in inverse problem theory are obtained for scalar Sturm-Liouville operators $\ell y = -y'' + q(x) y$ (see the monographs \cite{Mar77, Lev84, PT87, FY01}). Analysis of an inverse spectral problem usually includes the following steps.
\begin{enumerate}
    \item Uniqueness theorem.
    \item Constructive solution.
    \item Necessary and sufficient conditions of solvability.
    \item Local solvability and stability.
    \item Numerical methods.
\end{enumerate}

Uniqueness of solution in most cases is the simplest issue in inverse problem theory. Constructive solution usually means the reduction of a nonlinear inverse problem to a linear equation in a Banach space. In particular, the Gelfand-Levitan method \cite{Lev84} reduces inverse problems to Fredholm integral equations of the second kind. In the method of spectral mappings \cite{FY01, Yur02}, the main role is played by a linear equation in a space of infinite bounded sequences. By relying on these constructive methods, necessary and sufficient conditions of solvability have been obtained, local solvability and stability have been proved and also numerical techniques for solution \cite{RS92, IY08} have been developed for various classes of inverse spectral problems. We also have to mention the historically first method of Borg \cite{FY01, Borg46}. Initially, this method was local.
Further it was developed by Trubowitz \cite{PT87} and was used for investigating global solvability of inverse problems.

Matrix Sturm-Liouville operators in the form $\ell Y = -Y'' + Q(x) Y$, where $Q(x)$ is a matrix function, appeared to be more difficult for investigation. The main difficulties are caused by a complicated structure of spectral characteristics. Uniqueness issues of inverse problem theory for matrix Sturm-Liouville operators on finite intervals have been studied in \cite{Car02, Mal05, Chab04, Yur06-uniq, Sh07, Xu19}. In \cite{Yur06, Bond11}, a constructive solution, based on the method of spectral mappings, has been developed for the inverse problem for equation~\eqref{eqv} under the Robin boundary conditions
\begin{equation} \label{Robin}
Y'(0) - h Y(0) = 0, \quad Y'(\pi) + H Y(\pi) = 0,
\end{equation}
where $h, H \in \mathbb C^{m \times m}$, $h = h^{\dagger}$, $H = H^{\dagger}$. Chelkak and Korotyaev \cite{CK09} have given characterization of the spectral data
for the problem~\eqref{eqv}, \eqref{Robin} (in other words, necessary and sufficient conditions for the inverse problem solvability) in the case of asymptotically simple spectrum. In \cite{Bond11, Bond19-tamkang}, spectral data characterization has been obtained in the general case, with no restrictions on the spectrum,
for self-adjoint and non-self-adjoint eigenvalue problems in the form~\eqref{eqv},\eqref{Robin}. Analogous results for the Dirichlet boundary conditions
$Y(0) = Y(\pi) = 0$ are provided in \cite{Bond12}. Mykytyuk and Trush \cite{MT10} have given spectral data characterization for the matrix Sturm-Liouville operator with a singular potential from the class $W_2^{-1}$.

There are significantly less results on inverse matrix Sturm-Liouville problems with the general self-adjoint boundary conditions in the form~\eqref{bc}.
In \cite{Xu19}, several uniqueness theorems have been proved for such inverse problems on a finite interval. Harmer \cite{Har02, Har05} 
has studied inverse scattering on the half-line for the matrix Sturm-Liouville operator with the general self-adjoint boundary 
condition at the origin. However, we find inverse problems for matrix Sturm-Liouville operators on the half-line 
\cite{AM63, Har02, Har05, FY07, Bond15} and on the line \cite{CD77, Wad80, Olm85, AG98, Bond17-line} to be in some sense easier for investigation, than inverse problems on a finite interval. Namely, the spectrum of the matrix Sturm-Liouville operators on a finite interval can contain infinitely many groups of multiple or asymptotically multiple eigenvalues, while operators on infinite domains usually have a bounded set of eigenvalues.

The main goal of this paper is to develop a constructive algorithm for solving an inverse spectral problem for the matrix 
Sturm-Liouville operator~\eqref{eqv}-\eqref{bc}. In the future research, we plan to use this algorithm to obtain characterization 
of the spectral data, to investigate local solvability and stability of the considered inverse problem. 
As far as we know, all these issues have not been studied before for the operator in the form~\eqref{eqv}-\eqref{bc}. 
One can also develop numerical methods, basing on our algorithm. In addition, we intend to apply our results to inverse problems for differential operators on graphs \cite{Yur05, Yur16}.

Let us define the spectral data, used for reconstruction of the considered operator.
Denote $p := \rank(T)$ and assume that $1 \le p < m$. Then $\rank(T^{\perp}) = m - p$.
In \cite{Bond19}, asymptotic properties of the spectral characteristics of the problem \eqref{eqv}-\eqref{bc} have been studied. In particular, it has been proved that the spectrum of $L$ is a countable set of real eigenvalues $\{ \la_{nk} \}_{n \in \mathbb N, \, k = \overline{1, m}}$, such that
$$
\arraycolsep=1.4pt\def\arraystretch{1.5}
\left\{ 
\begin{array}{ll}
\sqrt{\la_{nk}} = n - \frac{1}{2} + O\left( n^{-1} \right),  & \quad k = \overline{1, p}, \\
\sqrt{\la_{nk}} = n + O\left( n^{-1}\right), & \quad 
k = \overline{p + 1, m},
\end{array}\right. \quad n \in \mathbb N.
$$
The more detailed eigenvalue asymptotics are provided in Section 2 of our paper.
Note that multiple eigenvalues are possible, and they occur in the sequence $\{ \la_{nk} \}_{n \in \mathbb N, \, k = \overline{1, m}}$ multiple times according to their multiplicities. One can also assume that $\la_{n_1, k_1} \le \la_{n_2, k_2}$, if $(n_1, k_1) < (n_2, k_2)$, i.e. $n_1 < n_2$ or $n_1 = n_2, \, k_1 < k_2$.

Let $\Phi(x, \la)$ be the $m \times m$ matrix solution of equation~\eqref{eqv}, satisfying the boundary conditions $\Phi(0, \la) = I_m$, $V(\Phi) = 0$. Define $M(\la) := \Phi'(0, \la)$. The matrix functions $\Phi(x, \la)$ and $M(\la)$ are called {\it the Weyl solution} and {\it the Weyl matrix} of $L$, respectively. The notion of the Weyl matrix generalizes the notion of the Weyl function for scalar Sturm-Liouville operators. Weyl functions and their generalizations play an important role in inverse problem theory for various classes of differential operators (see \cite{Mar77, FY01, Yur06, Bond19-tamkang}).
It can be easily shown that $M(\la)$ is a meromorphic matrix function. All the singularities of $M(\la)$ are simple poles, which coincide with the eigenvalues $\{ \la_{nk} \}_{n \in \mathbb N, \, k = \overline{1, m}}$.
Define the weight matrices 
\begin{equation} \label{defal}
\al_{nk} = -\Res_{\la = \la_{nk}} M(\la).
\end{equation}

The collection $\{ \la_{nk}, \al_{nk} \}_{n \in \mathbb N, \, k = \overline{1, m}}$ is called {\it the spectral data} of $L$. This paper is devoted to the following inverse problem.

\begin{ip} \label{ip:1}
Given the spectral data $\{ \la_{nk}, \al_{nk} \}_{n \in \mathbb N, \, k = \overline{1, m}}$, recover the problem $L$, i.e. find the potential $Q(x)$ and the matrices $T$, $H$.
\end{ip}

The statement of Inverse Problem~\ref{ip:1} generalizes the classical statement by Marchenko \cite{Mar77, FY01}.
Denote by $\{ \la_n \}_{n \ge 1}$ of the boundary value problem $-y'' + q(x) y = \la y$, $y(0) = y(\pi) = 0$, and by $\{ y_n(x) \}_{n \ge 1}$ the corresponding eigenfunctions, normalized by the condition $y_n'(0) = 1$. The inverse problem by Marchenko consists in recovering the potential $q(x)$ from the spectral data $\{ \la_n, \al_n \}_{n \ge 1}$, where $\al_n := \int_0^{\pi} y_n^2(x) \, dx$ are the so-called weight numbers or norming constants.
Spectral data similar to $\{ \la_{nk}, \al_{nk} \}_{n \in \mathbb N, \, k = \overline{1, m}}$ has been used for reconstruction of matrix Sturm-Liouville operators in~\cite{Yur06, CK09, Bond11, Bond19-tamkang, MT10} and other papers.

In order to solve Inverse Problem~\ref{ip:1}, we develop the ideas of the method of spectral mappings \cite{FY01, Yur02}, in particular, of its modification for matrix inverse Sturm-Liouville problems \cite{Yur06, Bond11, Bond19-tamkang, Bond16}. This method is based on the contour integration in the complex plane of the spectral parameter, so it is convenient for working with multiple and asymptotically multiple eigenvalues. Our key idea is to group the eigenvalues by asymptotics, and to use the sums of the weight matrices for each group. That allows us to construct a special Banach space of infinite bounded matrix sequences and to reduce Inverse Problem~\ref{ip:1} to a linear equation in this Banach space.

The paper is organized as follows. In Section~2, we provide preliminaries. In Section~3, the main equation of Inverse Problem~\ref{ip:1}
is derived, its unique solvability is proved, and a constructive algorithm for solving the inverse problems is developed. In Section~3,
we only outline the main idea of our method, while the technical details of the proofs are provided in Section~4.
In Section~5, we apply our results for the matrix Sturm-Liouville problem in the form~\eqref{eqv}-\eqref{bc} to the Sturm-Liouville operators
on a graph. In Section~6, our algorithm for solving Inverse Problem~\ref{ip:1} is illustrated by a numerical example.

\section{Preliminaries}

In this section, we provide asymptotic formulas for the eigenvalues and the weight matrices, obtained in \cite{Bond19}. We also state the uniqueness theorem for solution of Inverse Problem~\ref{ip:1}.

First we need some notations. Define the matrix
\begin{equation} \label{defOmega}
\Omega := \frac{1}{2} \int_0^{\pi} Q(x) \, dx,
\end{equation}
and the polynomials
$$
\mathcal P_1(z) = z^{p-m} \det(z I_m - T (\Omega - H) T), \quad
\mathcal P_2(z) = z^{-p} \det (z I_m - T^{\perp} H T^{\perp} ).
$$
One can easily show that $\mathcal P_1(z)$ and $\mathcal P_2(z)$ are polynomials of the degrees $p$ and $(m - p)$, respectively. Note that the matrices $T (\Omega - H) T$ and $T^{\perp} H T^{\perp}$ are Hermitian. Consequently, they have real eigenvalues, part of which coincide with the roots of $\mathcal P_j(z)$, $j = 1, 2$. Denote the roots of $\mathcal P_1(z)$ by $\{ z_k \}_{k = 1}^p$ and the roots of $\mathcal P_2(z)$ by $\{ z_k \}_{k = p + 1}^m$, counting with the multiplicities and in the nondecreasing order: $z_k \le z_{k + 1}$ for
$k = \overline{1, p-1}$ and $k = \overline{p + 1, m-1}$.

Denote $\rho_{nk} := \sqrt{\la_{nk}}$, $n \in \mathbb N$, $k = \overline{1, m}$. Without loss of generality, we assume that $\la_{nk} \ge 0$. This condition can be achieved by a shift of the spectrum. Then all the numbers $\rho_{nk}$ are real.

Below we use the matrix norm in $\mathbb C^{m \times m}$, induced by the Euclidean norm in $\mathbb C^m$, i.e. $\| A \| = \sqrt{\la_{max}(A^{\dagger} A)}$, where $\la_{max}$ is the maximal eigenvalue of a matrix. The symbol $C$ denotes various constants. The notation $\{ \varkappa_n \}$ is used for various sequences from $l_2$. The notation $\{ K_n \}$ is used for various matrix sequences, such that $\{ \| K_n \| \} \in l_2$.

The assertion of \cite[Theorem~2.1]{Bond19} implies the following asymptotics for the eigenvalues.

\begin{prop} \label{prop:asymptla}
The eigenvalues $\{ \la_{nk} \}_{n \in \mathbb N, \, k = \overline{1, m}}$ satisfy the asymptotic relations
\begin{equation} \label{asymptla}
\arraycolsep=1.4pt\def\arraystretch{1.8}
\left\{ 
\begin{array}{ll}
\rho_{nk} = n - \dfrac{1}{2} + \dfrac{z_k}{\pi (n - 1/2)}+ \dfrac{\varkappa_n}{n},  & \quad k = \overline{1, p}, \\
\rho_{nk} = n + \dfrac{z_k}{\pi n} + \dfrac{\varkappa_n}{n}, & \quad 
k = \overline{p + 1, m},
\end{array}\right. \quad n \in \mathbb N.
\end{equation}
(The sequence $\{ \varkappa_n \}$ may be different for different $k$).
\end{prop}

In order to provide asymptotics for the weight matrices $\{ \al_{nk} \}$, we need additional notations. In view of the definition~\eqref{defal}, if $\la_{n_1, k_1} = \la_{n_2, k_2}$, then $\al_{n_1, k_1} = \al_{n_2, k_2}$. Further we do not need to count such equal weight matrices multiple times. Therefore for every group of multiple eigenvalues $\la_{n_1, k_1} = \la_{n_2, k_2} = \ldots = \la_{n_l, k_l}$, $(n_1, k_1) < (n_2, k_2) < \ldots < (n_l, k_l)$, we define $\al'_{n_1, k_1} = \al_{n_1, k_1}$, $\al'_{n_j, k_j} = 0$, $j = \overline{2, l}$. It is supposed that there are exactly $l$ eigenvalues among $\{ \la_{nk} \}_{n \in \mathbb N, \, k = \overline{1, m}}$ equal to $\la_{n_1, k_1}$. Define the sums
\begin{gather*}
\al_n^{(s)} = \sum_{\substack{k = \overline{1, p} \\ z_k = z_s}} \al'_{nk}, \quad s = \overline{1, p}, \qquad
\al_n^{(s)} = \sum_{\substack{ k = \overline{p + 1, m} \\ z_k = z_s}} \al'_{nk}, \quad s = \overline{p + 1, m}, \\
\al_n^I = \sum_{k = 1}^p \al'_{nk}, \quad \al_n^{II} = \sum_{k = p + 1}^m \al'_{nk}.
\end{gather*}

The following proposition combines the results of Theorems~3.1 and~3.4 from \cite{Bond19}.

\begin{prop} \label{prop:asymptal}
The following asymptotic relations are valid:
\begin{gather} \label{asymptal1}
\al_n^I = \frac{2 (n - 1/2)^2}{\pi} \left( T + \frac{K_n}{n} \right), \quad \al_n^{II} = \frac{2 n^2}{\pi} \left( T^{\perp} + \frac{K_n}{n} \right), \\ \label{asymptal2}
\al_n^{(s)} = \frac{2 n^2}{\pi} (A^{(s)} + K_n), \quad s = \overline{1, m},
\end{gather}
where $n \in \mathbb N$, the matrices $A^{(s)} \in \mathbb C^{m \times m}$, $s = \overline{1, m}$, are uniquely specified by $T$, $\Omega$ and $H$. 
\end{prop}

In Proposition~\ref{prop:asymptal}, the certain formulas for the matrices $\{ A^{(s)} \}_{s = 1}^m$ are not provided, since they are unnecessary for the purposes of this paper. However, by virtue of \cite[Corollary~3.3]{Bond19}, the following relation holds:
\begin{equation} \label{relAs}
\sum_{s \in \mathcal S} z_s A^{(s)} = T (\Omega - H) T + T^{\perp} \Omega T^{\perp} =: \Theta,
\end{equation}
where
$$
\mathcal S := \{ s =\overline{1, p} \colon s = 1 \: \text{or} \: z_s \ne z_{s-1} \} \cup \{ s= \overline{p + 1, m} \colon s = p+1 \: \text{or} \: z_s \ne z_{s-1}\}. 
$$

By virtue of \cite[Theorem~3.1]{Xu19}, the weight matrices are Hermitian, non-negative definite: $\al_{nk} = \al^{\dagger}_{nk} \ge 0$, $n \in \mathbb N$, $k = \overline{1, m}$. This fact together with Proposition~\ref{prop:asymptal} yield the estimate
\begin{equation} \label{estal}
\al_{nk} = O\left( n^2\right), \quad n \in \mathbb N, \quad k = \overline{1, m}.
\end{equation}

Along with the problem $L$, we consider the problem $\tilde L = L(\tilde Q(x), \tilde T, \tilde H)$ of the same form \eqref{eqv}-\eqref{bc} as $L$, but with different coefficients $\tilde Q(x)$, $\tilde T$ and $\tilde H$. We agree that if a symbol $\ga$ denotes an object related to $L$, the symbol $\tilde \ga$ with tilde denotes the similar object related to $\tilde L$.

\begin{prop} \label{prop:asympt3}
Let two problems $L$ and $\tilde L$ be such that $T = \tilde T$ and $\Theta = \tilde \Theta$. Then 
\begin{gather*}
\rho_{nk} = \tilde \rho_{nk} + \frac{\varkappa_{n}}{n}, \quad
n \in \mathbb N, \quad k = \overline{1, m}, \quad \{ \varkappa_{n} \} \in l_2, \\
\al_n^I = \tilde \al_n^I + n K_n, \quad
\al_n^{II} = \tilde \al_n^{II} + n K_n, \quad
\al_n^{(s)} = \tilde \al_n^{(s)} + n^2 K_n, \quad
s = \overline{1, m}, \quad n \in \mathbb N.
\end{gather*}
\end{prop}

The next proposition is the uniqueness theorem for Inverse Problem~\ref{ip:1}.

\begin{prop} \label{prop:uniq}
Suppose that $\la_{nk} = \tilde \la_{nk}$ and $\al_{nk} = \tilde \al_{nk}$ for all $n \in \mathbb N$, $k = \overline{1, m}$. Then $Q(x) = \tilde Q(x)$ for a.a. $x \in (0, \pi)$, $T = \tilde T$, $H = \tilde H$.
\end{prop}

Proposition~\ref{prop:uniq} can be derived from the uniqueness results of Xu \cite{Xu19} for the matrix Sturm-Liouville operator 
with the two boundary conditions in the general self-adjoint form. Another way to prove Proposition~\ref{prop:uniq} is 
presented in Section~3. There we show that every solution of Inverse Problem~\ref{ip:1} corresponds to a solution of the main equation~\eqref{main}. Further the unique solvability of the main equation is proved, which implies Proposition~\ref{prop:uniq}.

\section{Inverse Problem Solution}

In this section, a constructive solution of Inverse Problem~\ref{ip:1} is obtained. We start with a choice of a model problem $\tilde L$
of the same form as $L$, but with different coefficients. Then Inverse Problem~\ref{ip:1} is reduced to the linear equation \eqref{mainS}, by using contour 
integration in the $\la$-plane (see Lemma~\ref{lem:contour}). Further we group the eigenvalues by asymptotics and introduce a special
Banach space $B$. It is shown that the linear equation \eqref{mainS} can be represented as the equation~\eqref{main} in $B$. Later on,
we prove the unique solvability of the main equation~\eqref{main} (see Theorem~\ref{thm:solve}). Further in this section, the solution
of the main equation is used for constructing $Q(x)$ and $H$ (see Theorem~\ref{thm:findQH}). Finally, we arrive at Algorithm~\ref{alg:ip}
for solving Inverse Problem~\ref{ip:1}. The results in this section are presented schematically. We provide auxiliary lemmas and the proofs
in the next section.

Let the spectral data $\{ \la_{nk}, \al_{nk} \}_{n \in \mathbb N, \, k = \overline{1, m}}$ of some unknown boundary value problem 
$L = L(Q(x), T, H)$ be given. Our goal is to construct the solution of Inverse Problem~\ref{ip:1}, i.e. to find $Q(x)$, $T$ and $H$.

Further we need a model problem $\tilde L = L(\tilde Q(x), \tilde T, \tilde H)$, satisfying the conditions $T = \tilde T$ and $\Theta = \tilde \Theta$. One can construct such model problem, by using the following algorithm.

\begin{alg} \label{alg:model}
Let the spectral data $\{ \la_{nk}, \al_{nk} \}_{n \in \mathbb N, \, k = \overline{1, m}}$ be given. We have to construct the model problem $\tilde L$.
\begin{enumerate}
\item Find $p$, relying on the eigenvalue asymptotics~\eqref{asymptla}.
\item Construct the matrices $\{ \al'_{nk} \}_{n \in \mathbb N, \, k = \overline{1, m}}$, $\{ \al_n^{(s)} \}_{n \in \mathbb N, \, s = \overline{1, m}}$ and $\{ \al_n^I \}_{n \in \mathbb N}$, by their definitions in Section~2.
\item Construct the matrices $T := \lim\limits_{n \to \iy} \frac{\pi}{2 n^2} \al_n^I$, $T^{\perp} := I_m - T$.
\item Find the numbers
\begin{align*}
& z_k = \lim_{n \to \iy} \left( \sqrt{\la_{nk}} - \Bigl( n - \tfrac{1}{2}\Bigr)\right) \pi \Bigl( n - \tfrac{1}{2}\Bigr), \quad k = \overline{1, p}, \\
& z_k = \lim_{n \to \iy} (\sqrt{\la_{nk}} - n) \pi n, \quad k = \overline{p + 1, m}.
\end{align*}
\item Construct the matrices
$$
A^{(s)} = \lim_{n \to \iy} \frac{\pi}{2 n^2} \al_n^{(s)}, \quad s \in \mathcal S.
$$
\item Calculate
$$
\Theta := \sum_{s \in \mathcal S} z_s A^{(s)}.
$$
\item Define $\tilde Q(x) \equiv \frac{2}{\pi} \Theta$, $x \in (0, \pi)$, $\tilde T := T$, $\tilde H := 0$, $\tilde L := L(\tilde Q(x), \tilde T, \tilde H)$.
\end{enumerate}
\end{alg}

The relation \eqref{relAs} guarantees that $\Theta = \tilde \Theta$. Consequently, the asymptotic relations of Proposition~\ref{prop:asympt3} hold for $L$ and $\tilde L$. 

Let us proceed to the derivation of the main equation. Denote by $S(x, \la)$ the matrix solution of equation~\eqref{eqv} under the initial conditions $S(0, \la) = 0$, $S'(0, \la) = I_m$. The following notation will be used for the matrix Wronskian: $\langle Z, Y\rangle = Z Y' - Z' Y$.
Define
\begin{equation} \label{defD}
   D(x, \la, \mu) = \frac{\langle S^{\dagger}(x, \bar \la), S(x, \mu) \rangle}{\la - \mu}.
\end{equation}

Introduce the notations 
\begin{gather*}
\la_{nk0} = \la_{nk}, \quad \la_{nk1} = \tilde \la_{nk}, \quad \rho_{nk0} = \rho_{nk}, \quad \rho_{nk1} = \tilde \rho_{nk}, \\
\al_{nk0} = \al_{nk}, \quad \al_{nk1} = \tilde \al_{nk}, \quad \al'_{nk0} = \al'_{nk}, \quad \al'_{nk1} = \tilde \al'_{nk}, \\
S_{nk0}(x) = S(x, \la_{nk0}), \quad S_{nk1}(x) = S(x, \la_{nk1}), \quad \tilde S_{nk0}(x) = \tilde S(x, \la_{nk0}), \quad 
\tilde S_{nk1}(x) = \tilde S(x, \la_{nk1}).
\end{gather*}

Using the method of spectral mappings \cite{FY01, Yur02}, we prove the following lemma.

\begin{lem} \label{lem:contour}
The following relations hold for $x \in [0, \pi]$, $n, r \in \mathbb N$, $k, q = \overline{1, m}$, $s, \eta = 0, 1$:
\begin{equation} \label{mainS}
\tilde S_{nks}(x) = S_{nks}(x) + \sum_{l = 1}^{\iy} \sum_{j = 1}^m (S_{lj0}(x) \al'_{lj0} \tilde D(x, \la_{lj0}, \la_{nks}) -
S_{lj1}(x) \al'_{lj1} \tilde D(x, \la_{lj1}, \la_{nks})),
\end{equation}\vspace*{-20pt}
\begin{multline} \label{contD}
    \al'_{rq\eta} \tilde D(x, \la_{rq\eta}, \la_{nks}) - \al'_{rq\eta} D(x, \la_{rq\eta}, \la_{nks}) = \sum_{l = 1}^{\iy} 
\sum_{j = 1}^m \bigl(\al'_{rq\eta} D(x, \la_{rq\eta}, \la_{lj0}) \al'_{lj0} \tilde D(x, \la_{lj0}, \la_{nks}) \\ - \al'_{rq\eta} D(x, \la_{rq\eta}, \la_{lj1}) \al'_{lj1} \tilde D(x, \la_{lj1}, \la_{nks}) \bigr).
\end{multline}
\end{lem}

\begin{proof}
Repeating the standard arguments of the proofs of \cite[Lemma~1.6.3]{FY01} and \cite[Lemma~1]{Yur06}, we derive the relation
\begin{equation} \label{intS}
\tilde S(x, \la) = S(x, \la) + \frac{1}{2\pi i}\int_{\ga} S(x, \mu) \hat M(\mu) \tilde D(x, \mu, \la) \, d\mu, \quad \hat M(\mu) := M(\mu) - \tilde M(\mu),
\end{equation}
where $\ga$ is the boundary of the region $X := \{ \la \colon \mbox{Re}\, \la > -\de, \,| \mbox{Im}\,\la | < \de\}$ with the counter-clockwise circuit, $\de > 0$ is a fixed number.
Clearly, under our assumptions, all the eigenvalues $\{ \la_{nk} \}_{n \in \mathbb N,\, k = \overline{1, m}}$ and $\{ \tilde \la_{nk} \}_{n \in \mathbb N,\, k = \overline{1, m} }$ lie in $X$.
The integral in \eqref{intS} converges for $\la \in \mathbb C \backslash X$ in the sense $\int_{\ga} := \lim_{N \to \iy} \int_{\ga_N}$, $\ga_N := \{ \la \in \ga \colon |\la| \le (N + 1/4)^2 \}$. 
Calculating the integral by the Residue Theorem, we obtain the relation
$$
\tilde S(x, \la) = S(x, \la) + \sum_{l = 1}^{\iy} \sum_{j = 1}^m (S_{lj0}(x) \al'_{lj0} \tilde D(x, \la_{lj0}, \la) -
S_{lj1}(x) \al'_{lj1} \tilde D(x, \la_{lj1}, \la)),
$$
where the series converges uniformly with respect to $x \in [0, \pi]$ and $\la$ on compact sets. Substituting~$\la = \la_{nks}$, we arrive at~\eqref{mainS}.

Similarly, following the proofs of \cite[Lemma~1.6.3]{FY01} and \cite[Theorem~4]{Yur06}, we derive the relation
$$
\tilde D(x, \la, \mu) - D(x, \la, \mu) = \frac{1}{2 \pi i} \int_{\ga} D(x, \la, \xi) \hat M(\xi) \tilde D(x, \xi, \mu) \, d\xi, 
$$
which implies~\eqref{contD}.
\end{proof}

For each fixed $x \in [0, \pi]$, the relation \eqref{mainS} can be considered as a system of linear equations with respect 
to $S_{nks}(x)$, $n \in \mathbb N$, $k = \overline{1, m}$, $s = 0, 1$. 
But the series in \eqref{mainS} converge conditionally in the following sense: $\lim\limits_{N \to \iy} \sum\limits_{l = 1}^N \sum\limits_{j = 1}^m (\ldots)$,
so it is inconvenient to use~\eqref{mainS} as a system of main equations of the inverse problem.
Below we transform \eqref{mainS} into an equation in a specially constructed Banach space of bounded infinite sequences.

We divide the square roots of the eigenvalues $\{ \rho_{nks} \}$ into collections by asymptotics. Put 
\begin{equation} \label{defG}
    G_1 = \{ \rho_{nks} \}_{n = \overline{1, n_0}, \, k = \overline{1, m}, \, s = 0, 1}, \quad
    G_{2j} = \{ \rho_{n_0+j,ks} \}_{k = \overline{1, p}, \, s = 0, 1}, \quad
    G_{2j + 1} = \{ \rho_{n_0 + j,ks} \}_{k = \overline{p + 1, m}, \, s = 0, 1},
\end{equation}
where $j \in \mathbb N$, and an integer $n_0$ is chosen so that $G_n \cup G_k = \varnothing$ for $n \ne k$. Such $n_0$ exists because of 
the asymptotic relations~\eqref{asymptla}. Each collection $G_n$ may contain multiple elements.

Consider a collection $\mathcal G = \{ g_i \}_{i = 1}^r$ of (possibly multiple) complex numbers.
Denote by $B(\mathcal G)$ the finite-dimensional space of all the matrix functions $f \colon \mathcal G \to \mathbb C^{m \times m}$, such that
$g_i = g_j$ implies $f(g_i) = f(g_j)$, with the norm
\begin{equation} \label{normBG}
   \| f \|_{B(\mathcal G)} = \max \left\{ \max_{i = \overline{1, r}} \| f(g_i) \|, \max_{\substack{g_i \ne g_j \\ i,j = \overline{1, r}}}
   |g_i - g_j|^{-1} \| f(g_i) - f(g_j) \| \right\}.
\end{equation}	

Introduce the Banach space $B$ of infinite sequences:
\begin{equation} \label{defB}
    B = \{ f = \{ f_n \}_{n \ge 1} \colon f_n \in B(G_n), \: \| f \|_{B} := \sup_{n \ge 1} (n \| f_n \|_{B(G_n)}) < \iy \}.
\end{equation}

For $x \in [0, \pi]$, we define the sequence 
\begin{equation} \label{defpsi}
    \psi(x) = \{ \psi_n(x) \}_{n \ge 1}, \quad \psi_n(x) = \{ \psi_n(x, \rho_{ljs}) , \: (l, j, s) \colon \rho_{ljs} \in G_n \},
\end{equation}
where $\psi_n(x, \rho) = S(x, \rho^2)$ for all $\rho \in G_n$.
Note that $S(x, \rho^2) \sim \frac{\sin \rho x}{\rho} I_m$ as $|\rho| \to \iy$. Using Schwarz's Lemma similarly to \cite[Section~1.6.1]{FY01},
we get the estimates
\begin{equation} \label{estpsi}
    \| \psi_n(x, \rho) \| \le \frac{C}{n}, \quad \| \psi_n(x, \rho) - \psi_n(x, \theta) \| \le \frac{C}{n} |\rho - \theta|, \quad
    n \in \mathbb N, \quad \rho, \theta \in G_n. 
\end{equation}
Hence $\| \psi_n(x) \| \le \frac{C}{n}$ for $n \in \mathbb N$, so $\psi(x) \in B$ for each $x \in [0, \pi]$.

For each fixed $x \in [0, \pi]$, we define the linear operator $R(x) \colon B \to B$, acting on any element $f = \{ f_n \}_{n \ge 1} \in B$
in the following way:
\begin{gather} \label{defR}
    (f R(x))_n = \sum_{k = 1}^{\iy} f_k R_{k, n}(x), \quad R_{k, n}(x) \colon B(G_k) \to B(G_n), \\ \label{defRkn}
    (f_k R_{k, n}(x))(\rho) = \sum_{(l, j) \colon \rho_{lj0}, \rho_{lj1} \in G_k} (f_k (\rho_{lj0}) \al'_{lj0} D(x, \rho^2_{lj0}, \rho^2) - 
    f_k(\rho_{lj1}) \al'_{lj1} D(x, \rho_{lj1}^2, \rho^2)), \quad \rho \in G_n.
\end{gather}

Thus, the action of the operator $R(x)$ is a multiplication of an infinite row vector $f$ of $m \times m$ matrices
by the infinite matrix.
In~\eqref{defR} and~\eqref{defRkn}, we put operators to the right of their operands, in order to keep the correct order of matrix multiplication.

\begin{thm} \label{thm:R}
The series in \eqref{defR} converges in $B(G_n)$-norm. For each fixed $x \in [0, \pi]$, the operator $R(x)$ is bounded and, moreover, compact on $B$.
\end{thm}

The proof of Theorem~\ref{thm:R} is provided in Section~4.

Define the element $\tilde \psi(x) \in B$ and the operator $\tilde R(x) \colon B \to B$ similarly to $\psi(x)$ and $R(x)$, respectively, 
with $\tilde S$, $\tilde D$ instead of $S$, $D$. Obviously, the relation~\eqref{mainS} can be rewritten
in the form
\begin{equation} \label{main}
  \tilde \psi(x) = \psi(x) (\mathcal I + \tilde R(x)),	
\end{equation}
where $\mathcal I$ is the identity operator in $B$. Clearly, the assertion of Theorem~\ref{thm:R} is valid for $\tilde R(x)$, i.e. for each fixed $x \in [0, \pi]$, the linear operator $\tilde R(x)$ is compact on $B$. We call equation~\eqref{main} in the Banach space $B$ {\it the main equation} of Inverse Problem~\ref{ip:1}.

\begin{thm} \label{thm:solve}
For each fixed $x \in [0, \pi]$, the main equation~\eqref{main} is uniquely solvable in the Banach space $B$.
\end{thm}

\begin{proof}
Using~\eqref{contD}, \eqref{defR} and~\eqref{defRkn}, we derive the relation
$$
(\mathcal I - R(x)) (\mathcal I + \tilde R(x)) = \mathcal I, \quad x \in [0, \pi].
$$
Symmetrically, one can obtain that
$$
(\mathcal I + \tilde R(x)) (\mathcal I - R(x)) = \mathcal I, \quad x \in [0, \pi].
$$
Therefore the operator $(\mathcal I + \tilde R(x))^{-1}$ exists and equals $(\mathcal I - R(x))$. By virtue of Theorem~\ref{thm:R}, the latter operator is bounded, so equation~\eqref{main} is uniquely solvable.
\end{proof}

By using the solution $\psi(x)$ of the main equation~\eqref{main}, one can construct the solution of Inverse Problem~\ref{ip:1}. Indeed, recall the definition~\eqref{defpsi} of $\psi(x)$. The known $\psi(x)$, in fact, gives us the sequence of the matrix functions $\{ S_{nks}(x) \}_{n \in \mathbb N, \, k = \overline{1,m}, \, s = 0, 1}$. These matrix functions satisfy the equation~\eqref{eqv} for $\la = \la_{nks}$, so one can construct the potential matrix by the formula:
$$
Q(x) = S''_{nks}(x) S^{-1}_{nks}(x) + \la_{nks} I_m.
$$
Then one can find $\Omega$ by \eqref{defOmega} and determine $H$ from~\eqref{relAs}:
$$
H = T \Omega T + T^{\perp} \Omega T^{\perp} - \Theta,
$$
since the matrices $\Theta$ and $T$ are already known (see Algorithm~\ref{alg:model}).

Below we describe another way to find $Q(x)$ and $H$. The following method is more convenient for further investigation of Inverse Problem~\ref{ip:1}, in particular, for characterization of the spectral data for the problem $L$.

Define the matrix functions
\begin{equation} \label{defeps}
    \eps_0(x) = \sum_{n = 1}^{\iy} \sum_{k = 1}^m (S_{nk0}(x) \al'_{nk0} \tilde S_{nk0}^{\dagger}(x) - S_{nk1}(x) \al'_{nk1} \tilde S_{nk1}^{\dagger}(x)), \quad
    \eps(x) = -2 \eps'_0(x).
\end{equation}

\begin{lem} \label{lem:eps}
The series \eqref{defeps} converges uniformly with respect to $x \in [0, \pi]$. Moreover, the matrix function $\eps_0(x)$ is absolutely continuous on $[0, \pi]$ and the elements of $\eps(x)$ belong to $L_2(0, \pi)$. 
\end{lem}

\begin{thm} \label{thm:findQH}
The following relations hold
\begin{equation} \label{findQH}
Q(x) = \tilde Q(x) + \eps(x), \quad H = \tilde H - T \eps_0(\pi) T.
\end{equation}
\end{thm}

The proofs of Lemma~\ref{lem:eps} and Theorem~\ref{thm:findQH} are provided in Section~4.
Finally, we arrive at the following algorithm for solving Inverse Problem~\ref{ip:1}.

\begin{alg} \label{alg:ip}
Let the spectral data $\{ \la_{nk}, \al_{nk} \}_{n \in \mathbb N, \, k = \overline{1, m}}$ be given. We have to construct $Q(x)$ and $H$.

\begin{enumerate}
\item Construct the model problem $\tilde L$, using Algorithm~\ref{alg:model}. At this step, we also determine the matrices $T$ and $\Theta$.
\item Find the matrix functions $\{ \tilde S_{nks}(x) \}_{n \in \mathbb N,\, k = \overline{1,m}, \, s = 0, 1}$ as the solutions of the initial value problems for equation~\eqref{eqv} with the potential $\tilde Q(x)$ and $\la = \la_{nks}$.
\item Using~\eqref{defD}, construct the functions $\tilde D(x, \la_{lj\eta}, \la_{nks})$ for $n, l \in \mathbb N$, $j, k = \overline{1, m}$, $\eta, s = 0, 1$.
\item Form the collections $\{ G_k \}_{k \ge 1}$ by \eqref{defG}.
\item Using the matrix functions $\tilde S_{nks}(x)$ and $\tilde D(x, \la_{lj\eta}, \la_{nks})$, form the element $\tilde \psi(x) \in B$ and the operator $\tilde R(x) \colon B \to B$ (see \eqref{defpsi}, \eqref{defR} and~\eqref{defRkn}).
\item Solve the main equation~\eqref{main} and find $\psi(x) \in B$, i.e. obtain $\{ S_{nks}(x) \}_{n \in \mathbb N, \, k = \overline{1, m}, \, s = 0, 1}$.
\item Construct $\eps_0(x)$ and $\eps(x)$ by \eqref{defeps}.
\item Find $Q(x)$ and $H$ by the formulas~\eqref{findQH}.
\end{enumerate}
\end{alg}

Algorithm~\ref{alg:ip} is theoretical. Relying on this algorithm, one can develop a numerical technique for solving Inverse 
Problem~\ref{ip:1}. For the scalar Sturm-Liouville equation ($m = 1$), the numerical algorithm, based on the method of spectral 
mappings, is provided in \cite{IY08}. Similarly one can obtain a numerical method for the matrix case, but this issue requires an additional investigation. In this paper, we illustrate the work of Algorithm~\ref{alg:ip} by a simple finite-dimensional example in Section~6. 

\section{Proofs}

In this section, the proofs of the assertions from Section~3 are provided. Our methods develop the approach of \cite{Bond16, Bond19-tamkang}
and are based on the grouping \eqref{defG} of the eigenvalues by their asymptotics.

\begin{lem} \label{lem:partG}
For collections $G_k$, $k \ge 1$, there are partitions into smaller collections
\begin{equation} \label{partGk}
    G_k = \cup_{i = 1}^{p_k} G_{ki}, \quad G_{ki} \cap G_{kj} = \varnothing, \quad i \ne j,
\end{equation}
such that
\begin{gather} \label{sumxi}
\sum_{k = 1}^{\iy} (k \xi_k)^2 < \iy, \\ \label{defxi}
\xi_k := \sum_{i = 1}^{p_k} \sum_{\rho, \theta \in G_{ki}} |\rho - \theta| + \frac{1}{k^3} \sum_{i = 1}^{p_k} \| \al(G_{ki}) - \tilde \al (G_{ki}) \| + \frac{1}{k^2} \| \al(G_k) - \tilde \al(G_k) \|.
\end{gather}
Here
$$
\al(\mathcal G) := \sum_{(l, j) \colon \rho_{lj0} \in \mathcal G} \al'_{lj0}, \quad \tilde \al(\mathcal G) := \sum_{(l, j) \colon \rho_{lj1} \in \mathcal G} \al'_{lj1},
$$
for any collection $\mathcal G$ of the form, described in Section~3.
\end{lem}

\begin{proof}
The assertion of the lemma immediately follows from Propositions~\ref{prop:asymptla} and~\ref{prop:asympt3}.
For $k = 1$, the partition is trivial: $p_1 = 1$, $G_{11} = G_1$.
For $k > 1$, each collection $G_{ki}$ is composed of the values with equal coefficients $z_s$ in the asymptotics~\eqref{asymptla}.
\end{proof}

\begin{lem} \label{lem:Schwarz}
For $n, k \in \mathbb N$, $\rho \in G_n$, $\theta, \chi \in G_k$, $x \in [0, \pi]$, the following estimates hold:
\begin{equation*}
    \| D(x, \theta^2, \rho^2) \| \le \frac{C}{nk (|n - k| + 1)}, \quad
    \| D(x, \theta^2, \rho^2) - D(x, \chi^2, \rho^2) \| \le \frac{C |\theta - \chi|}{nk (|n - k| + 1)},
\end{equation*}
where the constant $C$ does not depend on $n$, $k$, $\rho$, $\theta$, $\chi$ and $x$.
\end{lem}

\begin{proof}
This lemma is proved by the standard approach, based on Schwarz's Lemma (see \cite[Section~1.6.1]{FY01}). 
\end{proof}

\begin{lem} \label{lem:estRkn}
For $n, k \in \mathbb N$, $x \in [0, \pi]$, the following estimate holds:
\begin{equation} \label{estRkn}
    \| R_{k,n}(x) \|_{B(G_k) \to B(G_n)} \le \frac{C k \xi_k}{n (|n - k| + 1)},
\end{equation}
where the constant $C$ does not depend on $n$, $k$ and $x$.
\end{lem}

\begin{proof}
Fix $n, k \in \mathbb N$ and $x \in [0, \pi]$. Let $h$ be an arbitrary element of $G_k$. Put $\eta_{k,n} := h R_{k,n}(x)$. 
Clearly, $\eta_{k,n} \in G_n$.
Let us prove that
\begin{equation} \label{etakn}
    \| \eta_{k, n} \|_{B(G_n)} \le \frac{C k \xi_k \| h \|_{B(G_k)}}{n (|n - k| + 1)},
\end{equation}
where the constant $C$ does not depend on $n$, $k$, $x$ and $h$. Obviously, the estimate~\eqref{etakn} implies~\eqref{estRkn}.

According to \eqref{normBG}, we have
\begin{equation} \label{normeta}
    \| \eta_{k, n} \|_{B(G_n)} = \max \left\{ \max_{\rho \in G_n} \| \eta_{k, n}(\rho) \|, \, \max_{\substack{\rho \ne \theta \\ \rho, \theta \in G_n}} |\rho - \theta|^{-1} \| \eta_{k, n}(\rho) - \eta_{k, n}(\theta) \|\right\}.
\end{equation}

First we prove the estimate
\begin{equation} \label{esteta1}
\| \eta_{k, n}(\rho) \| \le \frac{C k \xi_k \| h \|_{B(G_k)}}{n (|n - k| + 1)}, \quad \rho \in G_n.
\end{equation}
In view of the definition~\eqref{defRkn} of $R_{k, n}(x)$, we have
$$
\eta_{k, n}(\rho) = \sum_{(l, j) \colon \rho_{lj0}, \rho_{lj1} \in G_k} (h(\rho_{lj0}) \al'_{lj0} D(x, \rho_{lj0}^2, \rho^2) - h(\rho_{lj1}) \al'_{lj1} D(x, \rho_{lj1}^2, \rho^2)).
$$
We derive that
$$
\eta_{k, n}(\rho) = J_1 + J_2 + J_3, 
$$
\begin{align*}
& J_1 := \sum_{(l, j) \colon \rho_{lj0}, \rho_{lj1} \in G_k} (h(\rho_{lj0}) - h(\rho_{lj1})) \al'_{lj0} D(x, \rho_{lj0}^2, \rho^2), \\
& J_2 := \sum_{(l, j) \colon \rho_{lj0}, \rho_{lj1} \in G_k} h(\rho_{lj1}) \al'_{lj1} (D(x, \rho_{lj0}^2, \rho^2) - D(x, \rho_{lj1}^2, \rho^2)), \\
& J_3 := \sum_{(l, j) \colon \rho_{lj0}, \rho_{lj1} \in G_k}
h(\rho_{lj1}) (\al'_{lj0} - \al'_{lj1}) D(x, \rho_{lj1}^2, \rho^2).
\end{align*}

Using~\eqref{normBG} for $h \in B(G_k)$ together with~\eqref{defxi}, we obtain the estimates
\begin{equation} \label{smeq1}
    \| h(\rho_{lj1}) \| \le \| h \|_{B(G_k)}, \quad
    \| h(\rho_{lj0}) - h(\rho_{lj1}) \| \le \xi_k \| h \|_{B(G_k)}, \quad \rho_{lj0}, \rho_{lj1} \in G_k.
\end{equation}
Lemma~\ref{lem:Schwarz} implies
\begin{equation} \label{smeq2}
    \| D(x, \rho_{lj0}^2, \rho) \| \le \frac{C}{nk (|n - k| + 1)}, \quad
    \| D(x, \rho_{lj0}^2, \rho^2) - D(x, \rho_{lj1}^2, \rho^2) \| \le \frac{C \xi_k}{nk (|n - k| + 1)}, 
\end{equation}
for all $\rho \in G_n$, $\rho_{lj0}, \rho_{lj1} \in G_k$.
It follows from \eqref{estal}, that 
\begin{equation} \label{smeq3}
    \| \al'_{ljs} \| \le C k^2, \quad \rho_{ljs} \in G_k.
\end{equation}
Using~\eqref{smeq1}, \eqref{smeq2} and~\eqref{smeq3}, we obtain the estimate
\begin{equation} \label{estJ}
    \| J_s \| \le \frac{C k \xi_k \| h \|_{B(G_k)}}{n (|n - k| + 1)}
\end{equation}
for $s = 1, 2$. It remains to prove~\eqref{estJ} for $J_3$.

Consider the partition~\eqref{partGk} of the collection $G_k$. For every $G_{ki}$, denote by $\rho_*(G_{ki})$ a fixed value $\rho_{lj1} \in G_{ki}$. We represent $J_3$ in the following form:
$$
J_3 = J_4 + J_5 + J_6,
$$
\vspace*{-30pt}
\begin{align*}
    & J_4 := \sum_{i = 1}^{p_k} \sum_{(l, j) \colon \rho_{lj1} \in G_{ki}} (h(\rho_{lj1}) - h (\rho_*(G_{ki}))) (\al'_{lj0} - \al'_{lj1}) D(x, \rho_{lj1}^2, \rho^2), \\
    & J_5 := \sum_{i = 1}^{p_k} \sum_{(l, j) \colon \rho_{lj1} \in G_{ki}} h(\rho_*(G_{ki})) (\al'_{lj0} - \al'_{lj1}) (D(x, \rho_{lj1}^2, \rho^2) - D(x, \rho_*^2(G_{ki}), \rho^2)), \\
    & J_6 := \sum_{i = 1}^{p_k} h(\rho_*(G_{ki})) (\al(G_{ki}) - \tilde \al(G_{ki})) D(x, \rho_*^2(G_{ki}), \rho^2).
\end{align*}

Using \eqref{normBG} and \eqref{defxi}, we get
\begin{equation} \label{smeq4}
    \| h (\rho_*(G_{ki})) \| \le \| h \|_{B(G_k)}, \quad \| h(\rho_{lj1}) - h(\rho_*(G_{ki})) \| \le  \xi_k \| h \|_{B(G_k)}, \quad \rho_{lj1} \in G_{ki}, \, i = \overline{1, p_k}.
\end{equation}
Lemma~\ref{lem:Schwarz} implies
\begin{gather} \label{smeq5}
    \| D(x, \rho_{lj1}^2, \rho^2) \| \le \frac{C}{nk (|n - k| + 1)}, \\ \label{smeq6}
    \| D(x, \rho_{lj1}^2, \rho^2) - D(x, \rho_*^2(G_{ki}), \rho^2) \| \le \frac{C \xi_k}{ nk (|n - k| + 1)},
\end{gather}
for all $\rho \in G_n$, $\rho_{lj1} \in G_{ki}$, $i = \overline{1, p_k}$. Combining~\eqref{smeq3}-\eqref{smeq6}, we conclude that~\eqref{estJ} holds for $s = 4, 5$.

In order to prove~\eqref{estJ} for $J_6$, we use the representation
$$
J_6 = J_7 + J_8 + J_9,
$$
\vspace*{-30pt}
\begin{align*}
    & J_7 := \sum_{i = 1}^{p_k} (h(\rho_*(G_{ki})) - h(\rho_*(G_{k1}))) (\al(G_{ki}) - \tilde \al(G_{ki})) D(x, \rho_*^2(G_{ki}), \rho^2), \\
    & J_8 := \sum_{i = 1}^{p_k} h(\rho_*(G_{k1})) (\al(G_{ki}) - \tilde \al(G_{ki})) (D(x, \rho_*^2(G_{ki}), \rho^2) - D(x, \rho_*^2(G_{k1}), \rho^2)), \\
    & J_9 := h(\rho_*(G_{k1})) (\al(G_k) - \tilde \al(G_k)) D(x, \rho_*^2(G_{k1}), \rho^2).
\end{align*}
Using~\eqref{asymptla} and Lemma~\ref{lem:Schwarz}, we obtain
\begin{gather} \label{smeq7}
    \| h(\rho_*(G_{ki})) - h(\rho_*(G_{k1})) \| \le |\rho_*(G_{ki}) - \rho_*(G_{k1})| \| h \|_{B(G_k)} \le \frac{C}{k} \| h \|_{B(G_k)}, \\ \label{smeq8}
    \| D(x, \rho_*^2(G_{ki}), \rho^2) - D(x, \rho_*^2(G_{k1}), \rho^2)\| \le \frac{C |\rho_*(G_{ki}) - \rho_*(G_{k1})|}{nk (|n - k| + 1)} \le \frac{C}{n k^2 (|n - k| + 1)},
\end{gather}
for $\rho \in G_n$, $i = \overline{1, p_k}$. Furthermore, it follows from~\eqref{defxi}, that
\begin{equation} \label{smeq9}
    \| \al(G_{ki}) - \tilde \al(G_{ki}) \| \le k^3 \xi_k, \: i = \overline{1, p_k}, \qquad
    \| \al(G_k) - \tilde \al(G_k) \| \le k^2 \xi_k.
\end{equation}

The estimates \eqref{smeq4}, \eqref{smeq5}, \eqref{smeq7}-\eqref{smeq9} together yield~\eqref{estJ} for $s = 7, 8, 9$. Consequently, the estimate~\eqref{esteta1} is valid.
\end{proof}

\begin{proof}[Proof of Theorem~\ref{thm:R}]
Fix $x \in [0, \pi]$ and suppose that $f = \{ f_n \}_{n \ge 1}$ is an arbitrary element of $B$. By virtue of~\eqref{defB}, we have
\begin{equation} \label{estfk}
    \| f_k \|_{B(G_k)} \le \frac{1}{k} \| f \|_B, \quad k \ge 1.
\end{equation}
The estimates~\eqref{estRkn} and~\eqref{estfk} imply
$$
\| f_k R_{k, n}(x) \|_{B(G_n)} \le \frac{C \xi_k \| f\|_B}{n (|n - k| + 1)}, \quad
n, k \ge 1.
$$
Using the latter estimate together with the definition~\eqref{defR}, we conclude that the series in~\eqref{defR} converges in $B(G_n)$ and
\begin{equation} \label{estfRn}
    \| (f R(x))_n \|_{B(G_n)} \le \sum_{k = 1}^{\iy} \| f_k R_{k, n}(x) \|_{B(G_n)} \le \frac{C \| f \|_{B}}{n} \sum_{k = 1}^{\iy} \xi_k.
\end{equation}
In view of~\eqref{sumxi}, we obtain
$$
\| (f R(x))_n \|_{B(G_n)} \le \frac{C \| f \|_B}{n}, \quad n \ge 1.
$$
According to~\eqref{defB}, we get $\| f R(x) \|_B \le C \| f \|_B$, i.e. the operator $R(x)$ is bounded on $B$.

Let us show that the operator $R(x)$ can be approximated by a sequence of finite-dimensional operators. For $s \ge 1$, define the operator $R^s(x) \colon B \to B$ as follows:
$$
R^s(x) = [R_{k, n}^s(x)]_{n, k = 1}^{\iy}, \quad
R_{k,n}^s(x) = \left\{ \begin{array}{ll}
                    R_{k, n}(x), \quad & k = \overline{1, s}, \\
                    0, \quad & k > s,
               \end{array} \right. \quad n \ge 1.
$$
Using~\eqref{estfRn}, one can easily show that
$$
\lim_{s \to \iy} \| R^s(x) - R(x) \|_{B \to B} = 0.
$$
Thus, the operator $R(x)$ is compact.
\end{proof}

\begin{remark}
Note that all the constants $C$ in the proof of Theorem~\ref{thm:R} do not depend on $x$.
\end{remark}

\begin{cor}
Define 
$$
\Lambda := \left( \sum_{k = 1}^{\iy} (k \xi_k)^2 \right)^{1/2}
$$
and fix $\Lambda_0 > 0$. If $\Lambda \le \Lambda_0$, the estimate $\| R(x) \|_{B \to B} \le C \Lambda$ holds, where the constant $C$ depends only on $\tilde L$ and $\Lambda_0$.
\end{cor}

\begin{proof}[Proof of Lemma~\ref{lem:eps}]
The definition \eqref{defeps} of $\eps_0(x)$ can be rewritten in the following form:
\begin{equation} \label{epsnew}
\eps_0(x) = \sum_{k = 1}^{\iy} \mathcal E_k(x), \quad
\mathcal E_k(x) := \sum_{(l, j) \colon \rho_{lj0}, \rho_{lj1} \in G_k} 
(S_{lj0}(x) \al'_{lj0} \tilde S^{\dagger}_{lj0}(x) - S_{lj1}(x) \al'_{lj1} \tilde S^{\dagger}_{lj1}(x)).
\end{equation}
The further arguments resemble the proof of Lemma~\ref{lem:estRkn}. Recall that every collection $G_k$ is divided into smaller collections $\{ G_{ki} \}_{i = 1}^{p_k}$ (see Lemma~\ref{lem:partG}). For every collection $G_{ki}$, we have chosen an arbitrary element $\rho_{lj1}$ and denoted it as $\rho_*(G_{ki})$. For brevity, denote the corresponding matrix function $S_{lj1}(x)$ by $S_*(x, G_{ki})$. Then we derive the relation
\begin{align} \nonumber
    \mathcal E_k(x) & = \sum_{(l, j) \colon \rho_{lj0}, \rho_{lj1} \in G_k} (S_{lj0}(x) - S_{lj1}(x)) \al'_{lj0} \tilde S^{\dagger}_{lj0}(x) \\ \nonumber 
    &  + \sum_{(l, j) \colon \rho_{lj0}, \rho_{lj1} \in G_k} S_{lj1}(x) \al'_{lj0} (\tilde S^{\dagger}_{lj0}(x) - \tilde S^{\dagger}_{lj1}(x)) \\ \nonumber & + \sum_{i = 1}^{p_k} \sum_{(l, j) \colon \rho_{lj1} \in G_{ki}} (S_{lj1}(x) - S_*(x, G_{ki})) (\al'_{lj0} - \al'_{lj1}) \tilde S^{\dagger}_{lj1}(x) \\ \nonumber
    & + \sum_{i = 1}^{p_k} \sum_{(l, j) \colon \rho_{lj1} \in G_{ki}} S_*(x, G_{ki}) (\al'_{lj0} - \al'_{lj1}) (\tilde S^{\dagger}_{lj1}(x) - \tilde S^{\dagger}_*(x, G_{ki})) \\ \nonumber
    & + \sum_{i = 1}^{p_k} (S_*(x, G_{ki}) - S_*(x, G_{k1})) (\al(G_{ki}) - \tilde \al(G_{ki})) \tilde S^{\dagger}_*(x, G_{ki}) \\ \nonumber
    & + \sum_{i = 1}^{p_k} S_*(x, G_{k1}) (\al(G_{ki}) - \tilde \al(G_{ki})) (\tilde S^{\dagger}_*(x, G_{ki}) - \tilde S^{\dagger}_*(x, G_{k1})) \\ \label{longE}
    & + S_*(x, G_{k1}) (\al(G_k) - \tilde \al(G_k)) \tilde S^{\dagger}_*(x, G_{k1}).
\end{align}

The relations~\eqref{estpsi} and~\eqref{defxi} yield
\begin{equation} \label{smeq10}
\left.
\begin{array}{c}
\| S_{ljs}(x) \| \le \frac{C}{k}, \quad
\| S_{lj0}(x) - S_{lj1}(x) \| \le \frac{C}{k} |\rho_{lj0} - \rho_{lj1}| \le \frac{C \xi_k}{k}, \quad \rho_{ljs} \in G_k, \\[10pt]
\| S_{lj1}(x) - S_*(x, G_{ki}) \| \le \frac{C \xi_k}{k}, \quad \rho_{lj1} \in G_{ki}, \: i = \overline{1, p_k}, \\[10pt]
\| S_*(x, G_{ki}) - S_*(x, G_{k1}) \| \le \frac{C}{k}, \quad i = \overline{1, p_k},
\end{array} \right\}
\end{equation}
where the constant $C$ does not depend on $k \in \mathbb N$ and on $x \in [0, \pi]$. The similar estimates are also valid for $\tilde S^{\dagger}$.

Using~\eqref{smeq3}, \eqref{smeq9}, \eqref{longE} and \eqref{smeq10}, we conclude that $\| \mathcal E_k(x) \| \le C \xi_k$, $k \in \mathbb N$, $x \in [0, \pi]$. Consequently, the series \eqref{epsnew} of continuous functions converges absolutely and uniformly with respect to $x \in [0, \pi]$, and
$$
\| \eps_0(x) \| \le C \sum_{k = 1}^{\iy} \xi_k \le C \Lambda.
$$

Next we show that the series $\sum_{k = 1}^{\iy} \mathcal E'_k(x)$ converges in $L_2$-norm. For definiteness, consider the first sum in~\eqref{longE}:
$$
Z_{1, k}(x) := \sum_{(l, j) \colon \rho_{lj0}, \rho_{lj1} \in G_k} (S_{lj0}(x) - S_{lj1}(x)) \al'_{lj0} \tilde S^{\dagger}_{lj0}(x).
$$
Differentiation yields
$$
    Z'_{1, k}(x) = \sum_{(l, j) \colon \rho_{lj0}, \rho_{lj1} \in G_k} \bigl( (S'_{lj0}(x) - S'_{lj1}(x)) \al'_{lj0} \tilde S^{\dagger}_{lj0}(x) 
    +  (S_{lj0}(x) - S_{lj1}(x)) \al'_{lj0} (\tilde S^{\dagger}_{lj0})'(x) \bigr)
$$
Further we use the asymptotic expressions
\begin{gather*}
S'_{lj0}(x) - S'_{lj1}(x) = -(\rho_{lj0} - \rho_{lj1}) x \sin (\rho_{lj0} x) I_m + O\left( \frac{\xi_k}{k}\right), \quad
\tilde S^{\dagger}_{lj0}(x) = \frac{\sin (\rho_{lj0} x)}{\rho_{lj0}} I_m + O\left( \frac{1}{k^2}\right), \\
S_{lj0}(x) - S_{lj1}(x) = \frac{(\rho_{lj0} - \rho_{lj1}) x \cos (\rho_{lj0} x)}{\rho_{lj0}} I_m + O\left( \frac{\xi_k}{k^2}\right), \quad
(\tilde S^{\dagger}_{lj0})'(x) = \cos (\rho_{lj0} x) I_m + O\left( \frac{1}{k} \right),
\end{gather*}
for $\rho_{lj0}, \rho_{lj1} \in G_k$. Here the $O$-estimates are uniform with respect to $x \in [0, \pi]$.
Taking the grouping \eqref{defG} into account, we define
$$
n_1 = 0, \quad n_{2j} = n_0 + j - 1/2, \quad n_{2j+1} = n_0 + j, \quad j \in \mathbb N. 
$$
Clearly,
$$
\rho_{lj0} = n_k + O\left( \frac{1}{k} \right), \quad \rho_{lj0} \in G_k,
$$
i.e. $n_k$ is the main part in the asymptotics~\eqref{asymptla} of the values from the collection $G_k$.
Finally, we get 
$$
Z'_{1,k}(x) = \Gamma_k x \sin(2 n_k x) + O(\xi_k), \quad \Gamma_k \in \mathbb C^{m \times m}, \quad k \in \mathbb N, \quad 
x \in [0, \pi], \quad \{ \| \Gamma_k \| \} \in l_2.
$$
Consequently, the elements of the matrix series $\sum_{k = 1}^{\iy} Z'_{1,k}(x)$ converge in $L_2(0, \pi)$. 
The similar technique can be applied to all the other terms in~\eqref{longE}. Thus, the elements of the matrix function $\eps(x)$ belong to $L_2(0, \pi)$. 
\end{proof}

\begin{proof}[Proof of Theorem~\ref{thm:findQH}]
{\bf Step 1.} First, we derive the relation for $Q(x)$. Using~\eqref{eqv} and~\eqref{defD}, we obtain $D'(x, \la, \mu) = S^{\dagger}(x, \bar \la) S(x, \mu)$. Then, formally differentiating~\eqref{intS} twice with respect to $x$, we get
\begin{multline*}
    \tilde S''(x, \la) = S''(x, \la) + \frac{1}{2\pi i} \int_{\ga} S''(x, \mu) \hat M(\mu) \tilde D(x, \mu, \la)\, d\mu  \\ + \frac{1}{\pi i} \int_{\ga} S'(x, \mu) \hat M(\mu) \tilde S^{\dagger} (x, \bar \mu) \, d\mu \, \tilde S(x, \la)
    + \frac{1}{2 \pi i} \int_{\ga} S(x, \mu) \hat M(\mu) \frac{d}{dx} (\tilde S^{\dagger}(x, \bar \mu) \tilde S(x, \la))\, d\mu.
\end{multline*}
Further we express the second derivatives from~\eqref{eqv} and use~\eqref{defD}, so we obtain
\begin{multline*}
    (\tilde Q(x) - \la) \tilde S(x, \la) = (Q(x) - \la) S(x, \la) + \frac{1}{2\pi i} \int_{\ga} (Q(x) - \mu) S(x, \mu) \hat M(\mu) \tilde D(x, \mu, \la) \, d\mu\\
    + 2 \frac{d}{dx} \left[ \frac{1}{2\pi i} \int_{\ga} S(x, \mu) \hat M(\mu) \tilde S^{\dagger}(x, \bar \mu) d\mu\right] \tilde S(x, \la) +
    \frac{1}{2 \pi i} \int_{\ga} S(x, \mu) \hat M(\mu) (\mu - \la) \tilde D(x, \mu, \la) \,d\mu
\end{multline*}
By Residue Theorem, the integral in the square brackets $[\ldots]$ equals $\eps_0(x)$, defined by~\eqref{defeps}. Consequently, taking~\eqref{intS} into account, we get the relation
$$
(\tilde Q(x) - Q(x)) \tilde S(x, \la) = 2 \eps_0'(x) \tilde S(x, \la),
$$
which implies $Q(x) = \tilde Q(x) + \eps(x)$.

{\bf Step 2.} Let us derive the relation for $H$ in \eqref{findQH}. Similarly to~\eqref{intS}, one can obtain the relation for the Weyl solution:
\begin{equation} \label{intPhi}
\tilde \Phi(x, \la) = \Phi(x, \la) + \frac{1}{2 \pi i} \int_{\ga} S(x, \mu) \hat M(\mu) \tilde E(x, \mu, \la) \, dx,
\end{equation}
where
\begin{equation} \label{defE}
\tilde E(x, \mu, \la) := \frac{\langle \tilde S^{\dagger}(x, \bar \mu), \tilde \Phi(x, \la) \rangle}{\mu - \la}, \quad
\tilde E'(x, \mu, \la) = \tilde S^{\dagger}(x, \bar \mu) \tilde \Phi(x, \la).
\end{equation}
Using~\eqref{intPhi} and~\eqref{defE}, we derive
\begin{equation} \label{VPhi}
V(\tilde \Phi) = V(\Phi) + \frac{1}{2\pi i} \int_{\ga} V(S(x, \mu)) \hat M(\mu) \tilde E(\pi, \mu, \la) \, d\mu + T \frac{1}{2\pi i} \int_{\ga} S(\pi, \mu) \hat M(\mu) \tilde S^{\dagger}(\pi, \bar \mu) \, d\mu \, \tilde \Phi(\pi, \la).
\end{equation} 

Recall that $V(\Phi) = 0$. Further it is shown that the first integral in~\eqref{VPhi} also equals zero. Indeed, the definition~\eqref{defE} yields
\begin{equation} \label{Epi}
\tilde E(\pi, \mu, \la) = \tilde S^{\dagger}(\pi, \bar \mu) (T + T^{\perp}) \tilde \Phi'(\pi, \la) - (\tilde S^{\dagger})'(\pi, \bar \mu) (T + T^{\perp})\tilde \Phi(\pi, \la).
\end{equation}
The projectors $T$ and $T^{\perp}$ are mutually orthogonal,
so it follows from the condition $\tilde V(\tilde \Phi) = 0$, that
$$
    T^{\perp} \tilde \Phi(\pi, \la) = 0, \quad T \tilde \Phi'(\pi, \la) = T \tilde H \tilde \Phi(\pi, \la).
$$
Consequently, the relation~\eqref{Epi} implies
\begin{equation} \label{Epi2}
\tilde E(\pi, \mu, \la) = \tilde S^{\dagger}(\pi, \bar \mu) T \tilde H \tilde \Phi(\pi, \la) + \tilde S^{\dagger}(\pi, \bar \mu) T^{\perp} \tilde \Phi'(\pi, \la) -  (\tilde S^{\dagger})'(\pi, \bar \mu) T \tilde \Phi(\pi, \la).
\end{equation}
Consider the linear form
$$
\tilde V^{\dagger} (\tilde S^{\dagger}(x, \bar \mu)) = ((\tilde S^{\dagger})'(\pi, \bar \mu) - \tilde S^{\dagger}(\pi, \bar \mu) \tilde H) T - \tilde S^{\dagger}(\pi, \bar \mu) T^{\perp}.
$$
One can show that the matrix functions $V(S(x, \mu)) M(\mu)$ and $\tilde M(\mu) \tilde V^{\dagger} (\tilde S^{\dagger}(x, \bar \mu))$ are entire. Consequently, the matrix functions $\tilde M(\mu) ((\tilde S^{\dagger})'(\pi, \bar \mu) - \tilde S^{\dagger}(\pi, \bar \mu) \tilde H) T$ and 
$\tilde M(\mu)\tilde S^{\dagger}(\pi, \bar \mu) T^{\perp}$ are also entire.
Therefore, in view of~\eqref{Epi2}, the first integral in~\eqref{VPhi} vanishes, so we get $V(\tilde \Phi) = T \eps_0(\pi) \tilde \Phi(\pi, \la)$. 

Obviously,
$$
V(\tilde \Phi) = \tilde V(\tilde \Phi) + T (H - \tilde H) \tilde \Phi(\pi, \la).
$$
Since $\tilde V(\tilde \Phi) = 0$, we obtain that 
$$
T (\tilde H - H - \eps_0(\pi) ) \tilde \Phi(\pi, \la) = 0.
$$
Using the asymptotic formula for the Weyl solution:
$$
\tilde \Phi(\pi, -\tau^2) = 2 T \exp(-\tau \pi) (1 + O(\tau^{-1})), \quad \tau \to +\iy,
$$
we conclude that $H = \tilde H - T \eps_0(\pi) T$.
\end{proof}

\section{Inverse Problem on the Star-Shaped Graph}

In this section, we apply our results to the Sturm-Liouville eigenvalue problem on the star-shaped graph \cite{Piv07, Bond19} in the form
\begin{equation} \label{eqvg}
-y_j'' + q_j(x) y_j = \la y_j, \quad x \in (0, \pi), \quad j = \overline{1, m},
\end{equation}
with the standard matching conditions
\begin{equation} \label{mc}
y_1(\pi) = y_j(\pi), \: j = \overline{2, m}, \qquad \sum_{j = 1}^m y_j'(\pi) = 0, \qquad
y_j(0) = 0, \: j = \overline{1, m},
\end{equation}
where $\{ q_j \}_{j = 1}^m$ are real-valued functions from $L_2(0, \pi)$.
Clearly, the boundary value problem~\eqref{eqvg}-\eqref{mc} can be rewritten in the form~\eqref{eqv}-\eqref{bc} with the diagonal matrix potential $Q(x) = \diag\{ q_j(x) \}_{j = 1}^m$, $H = 0$ and
$T = [T_{jk}]_{j,k = 1}^m$, $T_{jk} = \frac{1}{m}$, $j, k = \overline{1, m}$.

In this section, we use the notation $A^{(ii)} = a_{ii}$, $i = \overline{1, m}$, for diagonal elements of a matrix $A = [a_{jk}]_{j, k = 1}^m$. Yurko \cite{Yur05} has proved the uniqueness theorem and suggested an approach to solution of the following inverse problem.

\begin{ip} \label{ip:2}
Given the eigenvalues $\{ \la_{nk} \}_{n \in \mathbb N, \, k = \overline{1, m}}$ and the elements $\{ \al_{nk}^{(ii)} \}_{n \in \mathbb N, \, k = \overline{1, m}, \, i = \overline{1, m-1}}$ of the weight matrices, construct $\{ q_j \}_{j = 1}^m$.
\end{ip}

Thus, since the potential $Q(x)$ is diagonal, it is sufficient to use only the diagonal elements of the weight matrices, excluding the last elements $\{ \al_{nk}^{(mm)} \}_{n \in \mathbb N, \, k = \overline{1, m}}$. The method of Yurko consists of two steps.

\begin{alg} \label{alg:graph}
Let the data $\{ \la_{nk} \}_{n \in \mathbb N, \, k = \overline{1, m}}$ and $\{ \al_{nk}^{(ii)} \}_{n \in \mathbb N, \, k = \overline{1, m}, \, i = \overline{1, m-1}}$ be given. We have to construct $\{ q_i \}_{i = 1}^m$.

\begin{enumerate}
\item  Solving {\it local inverse problems}. For each $i = \overline{1, m-1}$, find $q_i(x)$, by using the data $\{ \la_{nk}, \al_{nk}^{(ii)} \}_{n \in \mathbb N, \, k = \overline{1, m}}$.
\item {\it Returning procedure}. Using the given data and the already constructed potentials $\{ q_i \}_{i = 1}^{m-1}$, find $q_m$.
\end{enumerate}
\end{alg}

For the first step of Algorithm~\ref{alg:graph}, Yurko suggested to derive the main equations in appropriate Banach spaces, 
by using the method of spectral mappings. Now we can easily obtain such main equations and prove their unique solvability, 
relying on the results of Section~3.

Diagonality of the matrix potential $Q(x)$ implies that the matrix functions $S(x, \la)$, $\tilde S(x, \la)$ and $\tilde D(x, \la, \mu)$ are also diagonal. Consequently, taking only the main diagonal in the system~\eqref{mainS}, we arrive at the scalar equations
\begin{equation} \label{mainSi}
\tilde S_{nks}^{(ii)}(x) = S_{nks}^{(ii)}(x) + \sum_{l = 1}^{\iy} \sum_{j = 1}^m (S_{lj0}^{(ii)} \al_{lj0}'^{(ii)} \tilde D^{(ii)}(x, \la_{lj0}, \la_{nks}) - S_{lj1}^{(ii)}(x) \al_{lj1}'^{(ii)} \tilde D^{(ii)}(x, \la_{lj1}, \la_{nks})),
\end{equation}
where $n \in \mathbb N$, $k = \overline{1, m}$, $s = 0, 1$. Equations~\eqref{mainSi} can be considered separately for each $i = \overline{1, m}$.

Denote the Banach space $B$ for $m = 1$ by $B_1$, i.e. $B_1$ is a space of scalar infinite sequences. For any element $f \in B$ and each $i = \overline{1, m}$, we can choose in every matrix component of $f$ the diagonal element at the position $(i, i)$ and obtain the element $f^{(ii)} \in B_1$, by combining these diagonal elements. Taking equation~\eqref{mainSi} into account, one can construct the compact linear operators $\tilde R^{(ii)}(x) \colon B_1 \to B_1$, $i = \overline{1, m}$, analogous to $\tilde R(x)$ and such that
\begin{equation} \label{maini}
\tilde \psi^{(ii)}(x) = \psi^{(ii)}(x) (\mathcal I_1 + \tilde R^{(ii)}(x)),
\quad i = \overline{1, m}, \quad x \in [0, \pi],
\end{equation}
where $\mathcal I_1$ is the identity operator in $B_1$. Analogously to Theorem~\ref{thm:solve}, we obtain the following result.

\begin{thm} \label{thm:graph}
For every $x \in [0, \pi]$ and each $i = \overline{1, m}$, the equation~\eqref{maini} is uniquely solvable in the Banach space $B_1$. 
\end{thm}

The main equations~\eqref{maini} can be used at step~1 of Algorithm~\ref{alg:graph} for solving the local inverse problems for $i = \overline{1, m-1}$. Theorem~\ref{thm:graph} justifies this step. Step~2 of Algorithm~\ref{alg:graph} is described in~\cite{Yur05}, and we do not elaborate into that issue in the present paper.

\section{Example}

In this section, we solve Inverse Problem~\ref{ip:1} for the following example.
Put $m = 3$. Consider the model problem $\tilde L = L(\tilde Q(x), \tilde T, \tilde H)$ with $\tilde Q(x) \equiv 0$, $\tilde H = 0$ and
\begin{equation} \label{defT}
\tilde T = \frac{1}{3} \begin{bmatrix} 
                    1 & 1 & 1 \\
                    1 & 1 & 1 \\
                    1 & 1 & 1
                \end{bmatrix}, \quad
\tilde T^{\perp} = I_3 - \tilde T = \frac{1}{3} \begin{bmatrix}
                    2 & -1 & -1 \\
                    -1 & 2 & -1 \\
                    -1 & -1 & 2
                \end{bmatrix}.
\end{equation}
This matrix Sturm-Liouville problem $\tilde L$ is equivalent to the Sturm-Liouville eigenvalue problem on the star-shaped graph \eqref{eqvg}-\eqref{mc} with $m = 3$ and $q_j(x) \equiv 0$, $j = \overline{1, 3}$.
It is easy to check, that this problem has the eigenvalues
$$
\tilde \la_{n1} = \left( n -\tfrac{1}{2} \right)^2, \quad 
\tilde \la_{n2} = \tilde \la_{n3} = n^2, \quad n \in \mathbb N,
$$
and the weight matrices
$$
\tilde \al_{n1} = \tfrac{2}{\pi} \left( n - \tfrac{1}{2}\right)^2 \tilde T, \quad
\tilde \al_{n2} = \tilde \al_{n3} = \tfrac{2}{\pi} n^2 \tilde T^{\perp}, \quad n \in \mathbb N.
$$
Suppose that the spectral data $\{ \la_{nk}, \al_{nk} \}_{n \in \mathbb N, \, k = \overline{1, 3}}$ of the problem $L$ differ from $\{ \tilde \la_{nk}, \tilde \al_{nk} \}_{n \in \mathbb N, \, k = \overline{1, 3}}$ only by the first eigenvalue:
\begin{gather} \nonumber
\la_{n1} = a^2, \quad a \in [0, 1), \: a \ne \tfrac{1}{2}, \\ \label{eqdata}
\la_{nk} = \tilde \la_{nk}, \: (n, k) \ne (1, 1), \qquad
\al_{nk} = \tilde \al_{nk}, \: n \in \mathbb N, \, k = \overline{1, 3}.
\end{gather}

Let us recover the potential matrix $Q(x)$ and the coefficient $H$ of the problem $L$ from its spectral data $\{ \la_{nk}, \al_{nk} \}_{n \in \mathbb N, \, k = \overline{1, 3}}$. Since the spectral data of the problems $L$ and $\tilde L$ coincide for $n \ge 2$, their asymptotics coincide, and therefore we can use the problem $\tilde L$ as the model problem for reconstruction of $L$ by the methods of Section~3. Moreover, we have $T = \tilde T$, $T^{\perp} = \tilde T^{\perp}$.

Consider the relation~\eqref{mainS}. Note that $S_{lj0}(x) \equiv S_{lj1}(x)$, if $\la_{lj0} = \la_{lj1}$. Consequently, we get that $S_{nks}(x) \equiv \tilde S_{nks}(x)$ for all $(n, k) \ne (1, 1)$, $s = 0, 1$. Hence we obtain from~\eqref{mainS} the system of two equations with respect to $S_{110}(x)$ and $S_{111}(x)$:
\begin{equation} \label{sys1}
\left\{ \begin{array}{l}
 \tilde S_{110}(x) = S_{110}(x) + S_{110}(x) \al_{11} \tilde D(x, \la_{110}, \la_{110}) - S_{111}(x) \tilde \al_{11} \tilde D(x, \la_{111}, \la_{110}), \\
\tilde S_{111}(x) = S_{111}(x) + S_{110}(x) \al_{11} \tilde D(x, \la_{110}, \la_{111}) - S_{111}(x) \tilde \al_{11} \tilde D(x, \la_{111}, \la_{111}).
\end{array}\right.
\end{equation}

For our example, we have
\begin{gather*}
\la_{110} = a^2, \quad \la_{111} = \tfrac{1}{4}, \quad \al_{11} = \tilde \al_{11} = \frac{1}{2\pi} T, \\
\tilde S_{110}(x) = \frac{\sin(a x)}{a} I_3, \quad \tilde S_{111}(x) = 2 \sin \left( \tfrac{x}{2} \right) I_3, \\
\tilde D(x, \la_{110}, \la_{110}) = \frac{1}{2a^2} \left( x - \frac{\sin 2 a x}{2 a}\right) I_3, \quad \tilde D(x, \la_{111}, \la_{111}) = 2 (x - \sin x) I_3, \\
\tilde D(x, \la_{110}, \la_{111}) = \tilde D(x, \la_{111}, \la_{110}) = 
\frac{1}{a} \left( \frac{\sin \bigl( a - \tfrac{1}{2}\bigr) x}{a - \tfrac{1}{2}} - \frac{\sin \bigl( a + \tfrac{1}{2}\bigr) x}{a + \tfrac{1}{2}} \right) I_3.
\end{gather*}
Consequently, the system~\eqref{sys1} takes the form
\begin{equation} \label{sys2}
\left\{ \begin{array}{l}
    S_{110}(x) (f_{11}(x) T + T^{\perp}) - S_{111}(x) f_{12}(x) T = \dfrac{\sin (a x)}{a} I \\[10pt]
    S_{110}(x) f_{12}(x) T + S_{111}(x) (f_{22}(x) T + T^{\perp}) = 2 \sin \left( \tfrac{x}{2}\right) I,
\end{array} \right.
\end{equation}
where
\begin{gather*}
    f_{11}(x) = 1 + \frac{1}{4\pi a^2} \left(x - \frac{\sin (2 a x)}{2a} \right), \quad f_{22}(x) = 1 - \frac{1}{\pi} (x - \sin x), \\
    f_{12}(x) = \frac{1}{2\pi a} \left( \frac{\sin \bigl(\bigl( a - \tfrac{1}{2}\bigr) x \bigr)}{a - \tfrac{1}{2}} - \frac{\sin \bigl( \bigl( a + \tfrac{1}{2}\bigr)x \bigr)}{a + \tfrac{1}{2}} \right).
\end{gather*}
Solving the system~\eqref{sys2}, we obtain
\begin{gather*}
S_{110}(x) =  \frac{\Delta_1(x)}{\Delta_0(x)} T + \frac{\sin (ax)}{a} T^{\perp}, \quad
S_{111}(x) =  \frac{\Delta_2(x)}{\Delta_0(x)} T + 2 \sin \left( \frac{x}{2}\right) T^{\perp}, \\
\Delta_0(x) := \begin{vmatrix} f_{11}(x) & -f_{12}(x) \\ f_{12}(x) & f_{22}(x) \end{vmatrix}, \quad
\Delta_1(x) := \begin{vmatrix} \tfrac{\sin (a x)}{a} & -f_{12}(x) \\ 2 \sin \left( \tfrac{x}{2}\right) & f_{22}(x) \end{vmatrix}, \quad 
\Delta_2(x) := \begin{vmatrix} f_{11}(x) & \tfrac{\sin (a x)}{a} \\ f_{12}(x) & 2 \sin \left( \tfrac{x}{2} \right)\end{vmatrix}.
\end{gather*}

It follows from \eqref{defeps}, \eqref{eqdata} and the above calculations, that
\begin{equation} \label{exeps}
\eps_0(x) = S_{110}(x) \al_{11} \tilde S_{110}^{\dagger}(x) - S_{111}(x) \tilde \al_{11} \tilde S_{111}^{\dagger}(x) = \frac{1}{2 \pi \Delta_0(x)}
\left( \frac{\Delta_1(x) \sin (a x)}{a} - 2 \Delta_2(x) \sin \bigl( \tfrac{x}{2}\bigr)\right) T.
\end{equation}
Then it is easy to find $Q(x)$ and $H$ by the formulas~\eqref{findQH}.

In particular, for $a = 0.3$, we have $Q(x) = q(x) T$, $H = h T$, where $h$ approximately equals $-0.361838$. The plot of $q(x)$ is presented in Figure~\ref{img:1}.

\begin{figure}[h!]
\begin{center}
\includegraphics[scale = 0.4]{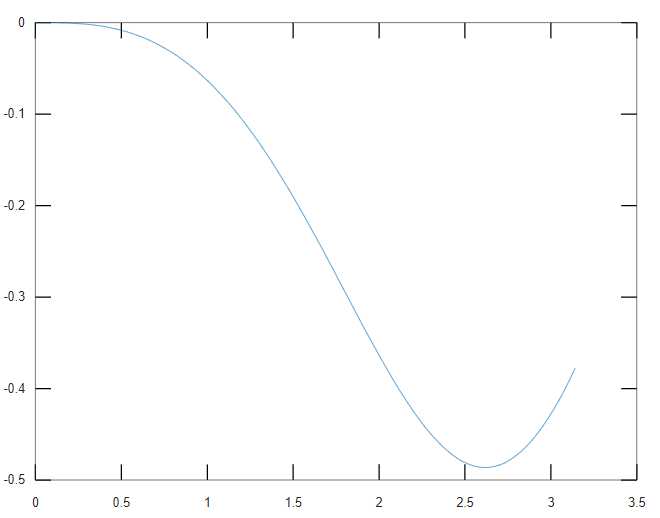}
\end{center}
\caption{Plot of $q(x)$, $x \in [0, \pi]$}
\label{img:1}
\end{figure}

Let us check our calculations, by finding the eigenvalues of the problem $L$. The solution $S(x, \la)$ has the form
$S(x, \la) = T s(x, \la) + T^{\perp} \rho^{-1} \sin \rho \pi$, where $s(x, \la)$ is the solution of the following scalar initial value problem with the constructed potential $q(x)$:
$$
-s''(x, \la) + q(x) s(x, \la) = \la s(x, \la), \quad s(0, \la) = 0, \quad s'(0, \la) = 1.
$$
The eigenvalues of the problem $L$ coincide with the zeros of its characteristic function
$$
\det V(S(x, \la)) = (s'(\pi, \la) - h s(\pi, \la)) \frac{\sin^2 \rho \pi}{\rho^2}.
$$
We have calculated the zeros of $(s'(\pi, \la) - h s(\pi, \la))$ numerically, using the forth-order Runge-Kutta method with the step $\frac{\pi}{1000}$. The zeros in the interval $[0, 50]$ are presented in the following table.
\begin{center}
    \begin{tabular}{|c|c|c|c|c|c|c|c|}
        \hline
         $n$ & 1 & 2 & 3 & 4 & 5 & 6 & 7 \\
         \hline
         $\lambda_{n1}$ & 0.090000 & 2.250000 & 6.250000 & 12.250000 & 20.250000 & 30.250000 & 40.250000 \\
         \hline
    \end{tabular}
\end{center}

Clearly, $\la_{11} = a^2$, $\la_{n1} = \left( n - \tfrac{1}{2}\right)^2$, $n \ge 2$. The multiplier $\frac{\sin^2 \rho \pi}{\rho^2}$ has the zeros $\la_{n2} = \la_{n3} = n^2$, $n \in \mathbb N$. Thus, the eigenvalues of the constructed problem $L$ coincide with the initially given values, and our method works correctly for this example.

It is clear from~\eqref{exeps}, that for $a \in [0, \tfrac{1}{2}) \cup (\tfrac{1}{2}, 1)$ the matrices $\eps_0(x)$ and $\tilde Q(x)$ are nondiagonal, so the problem $L$ is not the Sturm-Liouville problem on the star-shaped graph. This example shows that a simple perturbation of an eigenvalue can withdraw an operator out of the class of differential operators on graphs. In order to remain in this class, perturbations of the spectral data have to be connected with each other by additional conditions. Obtaining such conditions is a challenging topic for future research.

\medskip

{\bf Acknowledgment.} This work was supported by Grant 19-71-00009 of the Russian Science Foundation.

\medskip

\noindent Natalia Pavlovna Bondarenko \\
1. Department of Applied Mathematics and Physics, Samara National Research University, \\
Moskovskoye Shosse 34, Samara 443086, Russia, \\
2. Department of Mechanics and Mathematics, Saratov State University, \\
Astrakhanskaya 83, Saratov 410012, Russia, \\
e-mail: {\it BondarenkoNP@info.sgu.ru}

\end{document}